\documentclass{amsart}

\pdfoutput=1

\usepackage[utf8]{inputenc}
\usepackage[dvipsnames]{xcolor}
\usepackage{tikz}

\usepackage{graphics}
\usepackage{thmtools}
\usepackage[T1]{fontenc} 

\usepackage{amsthm}
\usepackage{amsbsy,amsmath,amssymb,amscd,amsfonts}

\usepackage[pagebackref=true]{hyperref}
\usepackage[nameinlink,capitalize,noabbrev]{cleveref}

\usepackage{graphicx,float,latexsym}
\usepackage[font={scriptsize,it}]{caption}
\usepackage{subcaption}

\usepackage{makecell}

\newtheorem{theorem}{Theorem}
\newtheorem*{theorem*}{Theorem}

\newtheorem{proposition}{Proposition}

\newtheorem{lemma}{Lemma}
\theoremstyle{remark}

\theoremstyle{definition}

\hypersetup{
    pdftoolbar=true,        % show Acrobat’s toolbar?
    pdfmenubar=true,        % show Acrobat’s menu?
    pdffitwindow=false,     % window fit to page when opened
    pdfstartview={FitH},    % fits the width of 
    colorlinks=true,       % false: boxed links; true: colored links
    linkcolor=OliveGreen,          % color of internal links (change box color with linkbordercolor)
    citecolor=blue,        % color of links to bibliography
    filecolor=black,      % color of file links
    urlcolor=red           % color of external links
}

\usepackage{lineno}

\usepackage{comment}
\usepackage{microtype}
\usepackage{footnote}

\newcommand{\E}{\mathcal{E}}

\def\cn{\mathrm{cn}}
\def\sn{\mathrm{sn}}

\title{Poncelet Plectra:\\Harmonious Curves in Cosine Space}
\author{Daniel Jaud}
\address{Holzkirchen Gymnasium\\
Munich, Germany}
\email{daniel.jaud.phd@gmail.com}
\author{Dan Reznik}
\address{Data Science Consulting\\Rio de Janeiro, Brazil}
\email{dreznik@gmail.com}

\author{Ronaldo Garcia}
\address{Inst. of Math. \& Statistics,
Federal Univ. of Goiás,
Goiânia, Brazil}
\email{ragarcia@ufg.br}

\begin{document}

\maketitle

\begin{abstract}
It has been shown that the family of Poncelet N-gons in the confocal pair (elliptic billiard) conserves the sum of cosines of its internal angles. Curiously, this quantity is equal to the sum of cosines conserved by its affine image where the caustic is a circle. We show that furthermore, (i) when N=3, the cosine triples of both families sweep the same planar curve: an equilateral cubic resembling a plectrum (guitar pick). We also show that (ii) the family of triangles excentral to the confocal family conserves the same product of cosines as the one conserved by its affine image inscribed in a circle; and that (iii) cosine triples of both families sweep the same spherical curve. When the triple of log-cosines is considered, this curve becomes a planar, plectrum-shaped curve, rounder than the one swept by its parent confocal family.
    
\medskip 
\noindent\textbf{Keywords:} elliptic billiard, invariant, confocal, Poncelet's porism, closure, conservation, cubic curve, spherical curve.

\medskip 
\noindent \textbf{MSC} {51N20, 51M04, 65-05}
\end{abstract}

\section{Introduction}
%\label{sec:intro}
Poncelet's porism is a remarkable result in projective geometry. In its simplest\footnote{The full result admits that the polygon sides be tangent to up to $N$ separate conics in the linear pencil defined by the outer one and a member of the second set  \cite[PCT]{centina15}.} form, it states that if an N-gon can be found inscribed to a first conic with all sides tangent to a second one, then a 1d family of such N-gons exists, parametrized by a vertex anywhere on the first conic; see  \cref{fig:poncelet}. 

\begin{figure}
    \centering
    \includegraphics[width=.65\textwidth]{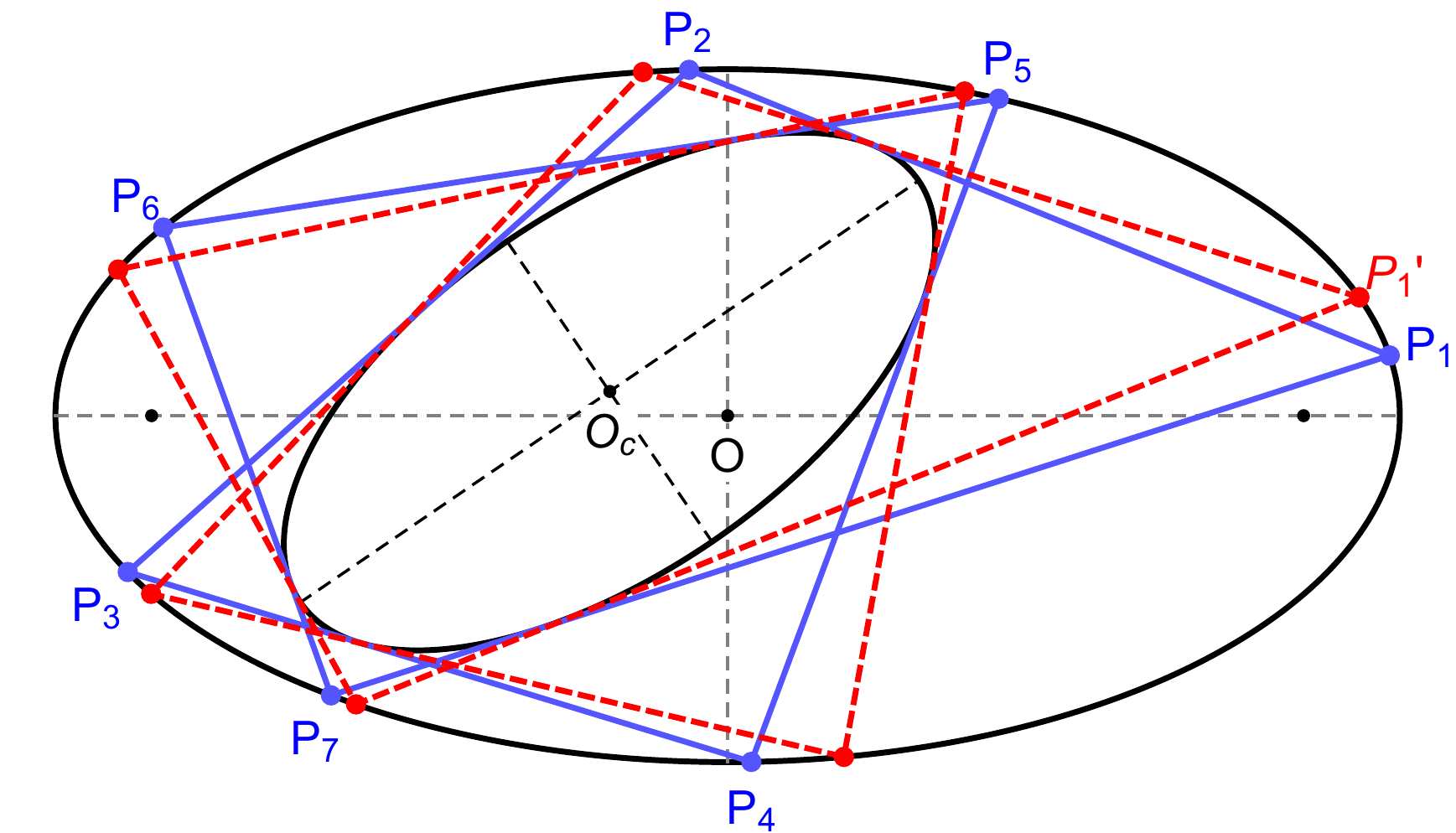}
    \caption{A self-intersected 7-gon (blue) is shown inscribed in an outer conic (center $O$) will all its sides tangent to an internal one (center $O_c$). Poncelet's porism states that a 1d family of such 7-gons exist, where one of the vertices (e.g., $P_1'$, corresponding to the dashed red polygon) can be any point of the outer conic. Videos:  \href{https://youtu.be/bjHpXVyXXVc}{N=3}, \href{https://youtu.be/kzxf7ZgJ5Hw}{N=7}}
    \label{fig:poncelet}
\end{figure}

Geometric details as well as fascinating historical notes on this result are covered in \cite{bos1987,centina15,rozikov2018}. Some introductory textbooks on this topic are \cite{rozikov2018,sergei91}. More advanced, algebro-geometric treatments include \cite{dragovic11,flatto2009,griffiths1978}.

Here we focus on special class of Poncelet N-gons: those interscribed\footnote{Shorthand for inscribed while simultaneously circumscribing.} to a pair of {\em confocal} ellipses; see \cref{fig:n5-trio}(left). Owing to Graves' theorem \cite{sergei91}, each vertex is bisected by the normal to the ellipse, i.e., the polygon can be regarded as the periodic trajectory of a point mass bouncing elastically against the outer ellipse, and therefore this system is also known as the {\em elliptic billiard}. Indeed, the latter is equipped with two integrals of motion (linear and angular momentum with respect to the foci), and is conjectured as the only integrable planar billiard \cite{kaloshin2018}. The two classical conservations are perimeter and Joachimsthal's constant \cite{sergei91}. 

\begin{figure}
    \centering
    \includegraphics[width=\textwidth]{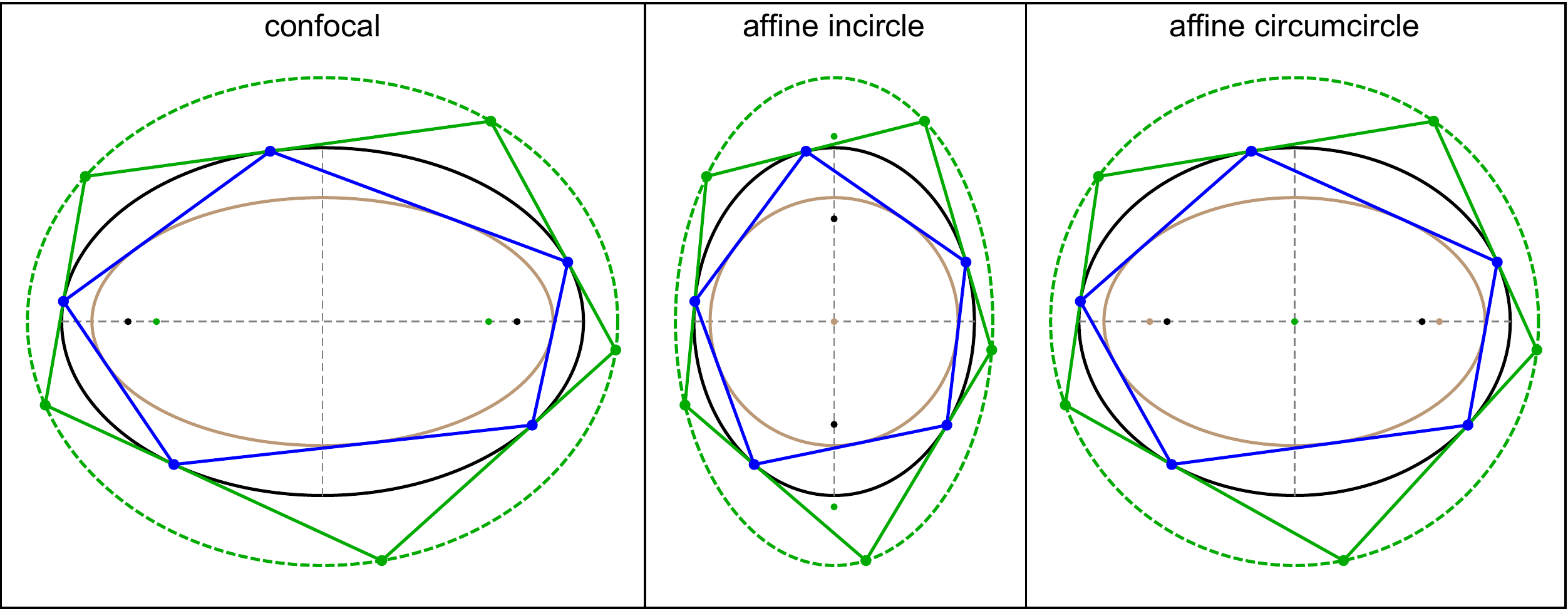}
    \caption{\textbf{Left:} An elliptic billiard Poncelet 5-periodic (blue), i.e., interscribed in a confocal pair of ellipses (black and brown). For all $N$, the 1d Poncelet family conserves the sum of its internal cosines. Also shown is the ``outer'' polygon (green), whose sides are tangent to the outer ellipse at the billiard polygon vertices. The Poncelet grid \cite{schwartz2007-grid,levi2007-poncelet-grid}, ensures outers are also inscribed in an ellipse (dashed green). These conserve the {\em product} of internal angle cosines. \textbf{Middle:} The affine image of billiard N-periodics which sends the confocal caustic to a circle (brown). Their sum of cosines is also conserved and {\em equal} to that of the confocal pre-image (left). \textbf{Right:} The affine image of confocals (left) which sends the ellipse circumscribing outer polygons (dashed green) to a circle. This family conserves also conserves its product of cosines and it is equal to that of the outers in its affine pre-image (left). Videos:
    \href{https://youtu.be/aWs29dqY34M}{N=3},
    \href{https://youtu.be/buhDTsjpvRQ}{N=5}}
    \label{fig:n5-trio}
\end{figure}

Dozens of other curious, metric conservations of elliptic billiard $N$-periodics (as these polygon families are called), have been experimentally detected \cite{reznik2021-fifty}. Naturally, these are functionally dependent on the classical ones, though in most cases one cannot find an expression for the dependency. 

Here we return to two early invariants experimentally detected in \cite{reznik2020-intelligencer}, and proved for all $N$  \cite{akopyan2020-invariants,bialy2020-invariants}:

\begin{enumerate}
    \item The sum of internal angle cosines.
    \item The {\em product} of internal angle cosines of the so-called ``outer'' family, i.e., polygons whose sides are tangent to the billiard at the N-periodic vertices.
\end{enumerate}

Additionally, though not addressed here, when $N$ is odd, the outer-to-N-periodic area ratio is invariant, a result proved in \cite{caliz2020-area-product}.

Referring to \cref{fig:n5-trio}(middle), consider the affine image of the confocal family which sends the caustic to a circle. Interestingly, one observes that the sum of cosines of this family is also constant. Surprisingly, the conserved quantity is numerically equal to that of its confocal pre-image. A. Akopyan has remarked that the first (or even both observations) can be regarded as a corollary to \cite[Theorem 6.4]{akopyan2020-invariants}.

Owing to the Poncelet grid \cite{levi2007-poncelet-grid,schwartz2007-grid}, the outer family is inscribed in an ellipse; see \cref{fig:n5-trio}(left). Consider the affine image of the confocal family which sends said outer ellipse to a circle\footnote{The affine image which sends the elliptic billiard to a circle also produces an N-periodic family whose invariant cosine product is equal to that conserved by the outer polygons of the confocal pre-image.}, as in \cref{fig:n5-trio}(right). It turns out the new circle-inscribed outer family will also conserve the product of internal cosines. Curiously, the conserved product is numerically identical to the excentral cosine product in the confocal pre-image.

Here we explore the above phenomena for the $N=3$ case, still amenable to treatment by analytic geometry. $N=4$ is a degenerate (and simpler) case: trajectories are parallelograms \cite{connes07}, the sum of cosines is zero, and the outer polygon is a rectangle inscribed in Monge's orthoptic circle \cite{reznik2020-intelligencer}. The $N=5$ case, shown in \cref{fig:n5-trio}, turns out to be unwieldy: the ellipse parameters are constrained by a ternary cubic (see \cite[Eqn. 5]{mw}). In practice, these are computed via numerical optimization, see \cite[Section 4]{reznik2021-fifty}. To properly analyze $N=5$ and higher, we would need algebro-geometric techniques such as the ones used in \cite{akopyan2020-invariants}, though these are beyond our reach. Indeed, $N{\geq}5$ porisms are typically parametrized with Jacobi elliptic functions, see for example \cite{stachel2021-motion}.

\subsection*{Main Results}

The 3-periodic confocal family and its affine image with incircle are shown side-by-side in \cref{fig:n3-trio} and superposed in \cref{fig:two-families}. Besides conserving the same sum of cosines, we show both families sweep the same planar curve in the space of cosine triples, called here ``cosine space'' (this fact is not yet proved for $N>3$), and that this curve is an equilateral cubic in the shape of a plectrum, see \cref{fig:picks-3d,fig:proj-triples}.

Similarly, the $N=3$ family of outer polygons (i.e., their excentral triangles) is shown superposed to its affine image with fixed circumcircle in \cref{fig:two-families-exc}. As mentioned above, these conserve the same product of cosines. We further show that either family sweeps the same curve in cosine space, and that this curve lies on the intersection of a sphere with the so-called Titeica surface  \cite{ferreol2017-titeica}; see  \cref{fig:sfiha,fig:titeica}.

In ``log-cosine'' space the spherical curve is flattened to a planar curve also resembling a plectrum though rounder than the one swept by the confocal-incircle duo; see \cref{fig:superposed}.

Below we (i) derive the affine transformations required to send the confocal (resp.\ outer, i.e., excentral) family to one with a fixed incircle (resp.\ circumcircle); (ii) prove the cosine sums (resp.\ product) of original and affine image are indeed the same; and (iii) compute explicit expressions for the curves swept in cosine space for either case, i.e., an equilateral cubic (resp.\ a spherical curve). Some figures contain links to animations, which are compiled in a table at the end. We close with a few questions for the reader.

\begin{figure}
    \centering
    \includegraphics[width=\textwidth]{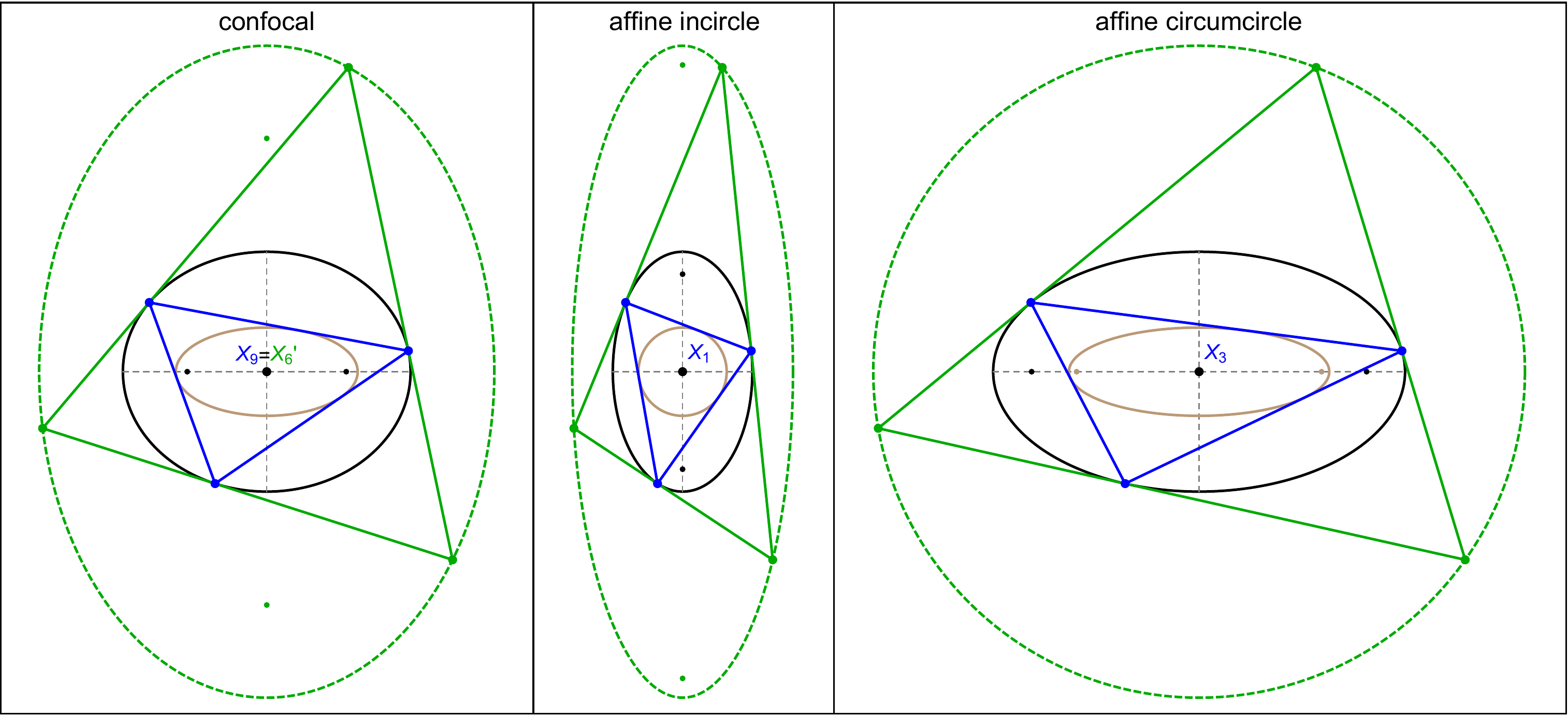}
    \caption{As before the trio of confocal 3-periodics (left), their affine image with fixed incircle (middle), and their affine image with excentral fixed circumcircle (right). \href{https://youtu.be/aWs29dqY34M}{Video}}
    \label{fig:n3-trio}
\end{figure}

\subsection*{Related Work} Trajectories in a ``bar'' billiard (comprising two circles) are studied in \cite{cieslak2013}. A billiard with reflections approximated by smooth functions is studied in \cite{lynch2019}. The recent trilogy \cite{stachel2021-motion,stachel2021-iso,stachel2021-grid} derives several novel metric properties of trajectories in elliptic billiards.

\subsection*{Article Organization} In \cref{sec:inc-conf} (resp. \cref{sec:circ-exc}) we analyze cosine sum conservation for confocals and their affine image with incircle (resp. cosine product conservation for confocal-excentrals and their affine image with circumcircle). In \cref{sec:cos-space-inc-conf,sec:cos-space-circ-exc} we analyze the common space curves swept by cosine triples in either case. In \cref{sec:conclusion} we tabulate all videos mentioned in figure captions, concluding with several open questions to the reader.

\section{Incircle and Confocals: Equal Sum of Cosines}
\label{sec:inc-conf}
Consider a pair of concentric, axis-aligned ellipses $\E,\E_c$ with semi-axes $a,b$ and $a_c,b_c$, respectively. The Cayley condition for these to admit a 1d family of Poncelet 3-periodics reduces to \cite{georgiev2012-cayley}:

\begin{equation}
    \frac{a_c}{a}+\frac{b_c}{b} = 1 
    \label{eq:cayley}
\end{equation}

Consider the family of Poncelet 3-periodics inscribed in an ellipse and circumscribed about a concentric circle of radius $r$, i.e.,  $a_c=b_c=r$. We call them the ``incircle family''. In this case, \cref{eq:cayley} implies $r=(a b)/(a+b)$.

Let the ``confocal family'' denote Poncelet 3-periodics interscribed in a pair of confocal ellipses. Referring to \cref{fig:two-families}:

\begin{lemma}
The confocal family is the image of the incircle family under a scaling by $s$ along the major axis, where:

\[ s = \sqrt{\frac{{b^4+2 a b^3}}{{a^4+2 b a^3}}} \]

%\[ \lambda =\frac{(2a-r)\cdot r^3}{(a^2-r^2)\cdot (a-r)^2} \]
\label{lem:scaling}
\end{lemma}

\begin{proof}
The incircle will be inscribed in an ellipse with semi-axes $s a, b$. By Cayley, be an ellipse with semi-axes $a_c,b_c$ given by:

\[ a_c=s \frac{a b}{a + b},\;\;\;b_c=\frac{ab}{a + b} \]

The claim is obtained by imposing $a_c^2-b_c^2=(s a)^2-b^2$ and solving for $s^2$.
\end{proof}

\begin{proposition}
Both the incircle family and its $s$-scaled confocal image (\cref{lem:scaling}) have identical sums of cosines $k$ given by:

\[k=\sum_{i=1}^3 \cos\theta_i=1+\frac{2 a b}{(a+b)^2}=1+\frac{2r(a-r)}{a^2}\]
\label{prop:cos-sum}
\end{proposition}

\begin{proof}
 The incircle family conserves circumradius given by $R=(a+b)/2$ \cite[Thm. 1]{garcia2020-family-ties}. 
From \cref{eq:cayley}, $r = (ab)/(a+b)$. So the incircle family has invariant $r/R=(2 a b)/(a+b)^2$. In \cite[Thm. 1]{garcia2020-new-properties} an expression for the invariant ratio $(r/R)_\text{conf}$ over confocal 3-periodics is provided, where $\alpha,\beta$ are the semi-axes of the outer ellipse in the pair:

\begin{equation}
(r/R)_\text{conf}=\frac{2 (\delta-\beta^2)(\alpha^2-\delta)}{(\alpha^2-\beta^2)^2},\;\;\;\delta = \sqrt{\alpha^4-\alpha^2 \beta^2+\beta^4}
\label{eqn:rOvR}
\end{equation}

By setting $\alpha=s\,a$ and $\beta=b$ in the above, where $s$ is as in \cref{lem:scaling}, after simplification one obtains $(r/R)_\text{conf}=(2 a b)/(a+b)^2$. Recall the sum of cosines is equal to $1+r/R$ \cite[Inradius, Eq. 9]{mw}.
\end{proof}

Recall that as stated in \cref{sec:inc-conf}, the incircle and its affine confocal image have identical sums of cosines for all $N$ \cite[Corollary 6.4]{akopyan2020-invariants}.

\begin{figure}
    \centering
    \includegraphics[width=.7\textwidth]{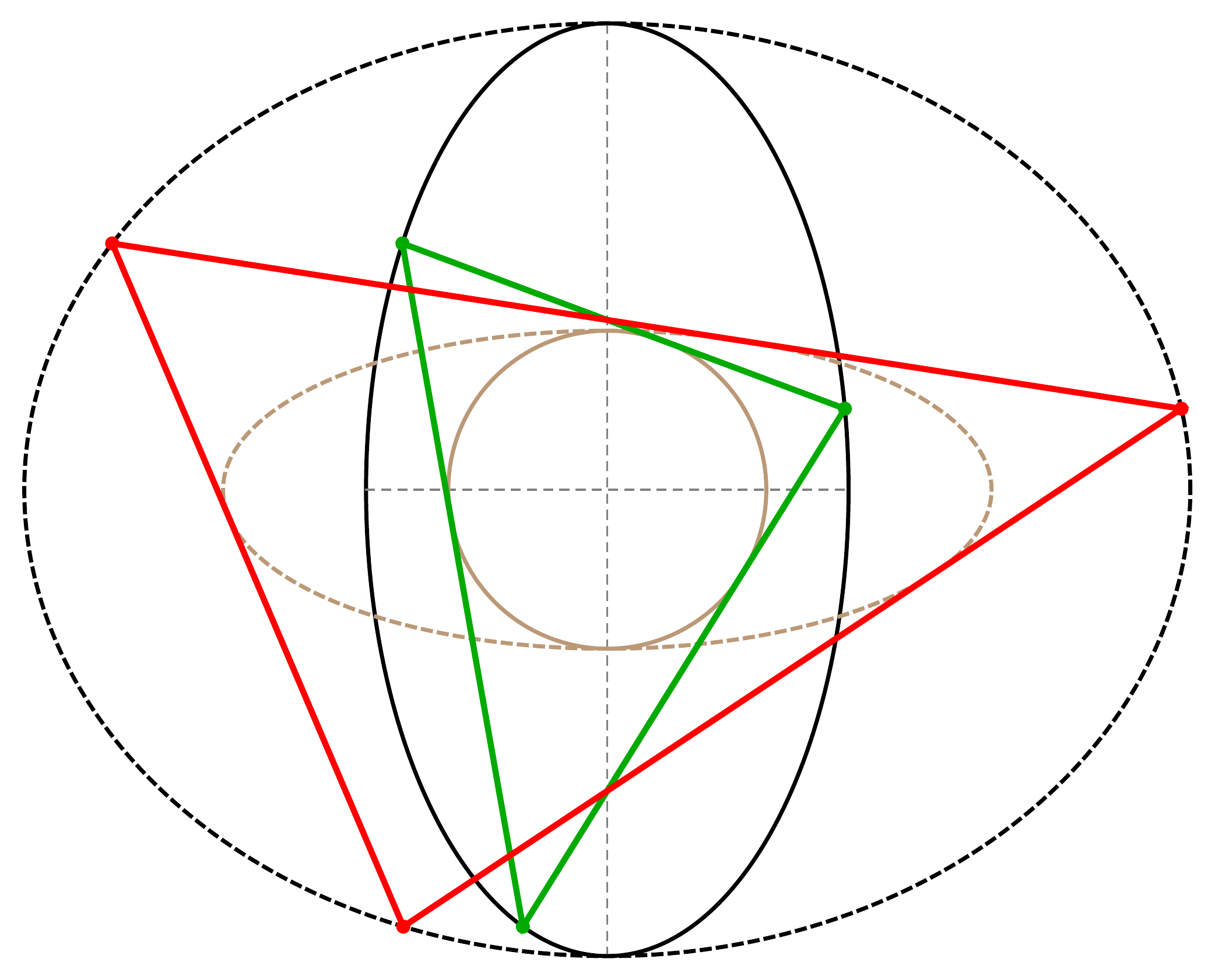}
    \caption{A 3-periodic (green) of the incircle family (solid black and brown ellipses) is shown along with a 3-periodic (red) in the confocal family (dashed black and brown) such that the latter is a scaled image of the former along the major axis. Their angle triples sweep identical curves (\cref{thm:incircle-confocal-space}).
    \href{https://youtu.be/CKVoQvErjj4}{Video}}
    \label{fig:two-families}
\end{figure}

\section{Circumcircle and Excentrals: Equal Product of Cosines}
\label{sec:circ-exc}
Consider the Poncelet family of 3-periodics inscribed in a circle of radius $R$ and circumscribing a concentric ellipse of semi-axes $a,b$. In this case, \cref{eq:cayley} implies $R=a+b$. We refer to these as the ``circumcircle family''.

Let the ``excentral family'' denote the excentral triangles to confocal 3-periodics, i.e., whose sides pass through the vertices perpendicular to the bisectors \cite[Excentral Triangle]{mw}. Referring to \cref{fig:two-families-exc}:

\begin{lemma}
A unique scaling transformation sends the circumcircle family to the excentral one. This is a scaling by $s'$ along the circumcircle's caustic major axis, where:

\[ s' = \sqrt{\frac{{2b^2+a b}}{{2 a^2+a b}}} \]
\label{lem:scaling-exc}
\end{lemma}

\begin{proof}
Let $a,b$ denote the axes of caustic to the circumcircle family. Since $R = (a + b)$, a scaling by $s'$ along the major axis sends the outer circle to $(s'(a+b),a+b))$ and the caustic to $(s' a,b)$.  In \cite{garcia2019-incenter} a formula was derived for the semi-axes $a_e,b_e$ of the elliptic locus of the excenters of 3-periodics in the confocal pair where $\alpha,\beta$ are its major and minor semi-axes:

\[ a_e=(\beta^2+\delta)/\alpha,\;\;\;b_e=(\alpha^2+\delta)/\beta \]
where as before, $\delta=\sqrt{\alpha^4-\alpha^2 \beta^2+\beta^4}$.

Since the excentral family is the family of excentral triangles to confocal 3-periodics, we set $\alpha=s' a$ and $\beta = b$ and impose $b_e=a+b$, solving this for $s'$. Simplification yields the result.
\end{proof}

\begin{figure}
    \centering
    \includegraphics[width=.7\textwidth]{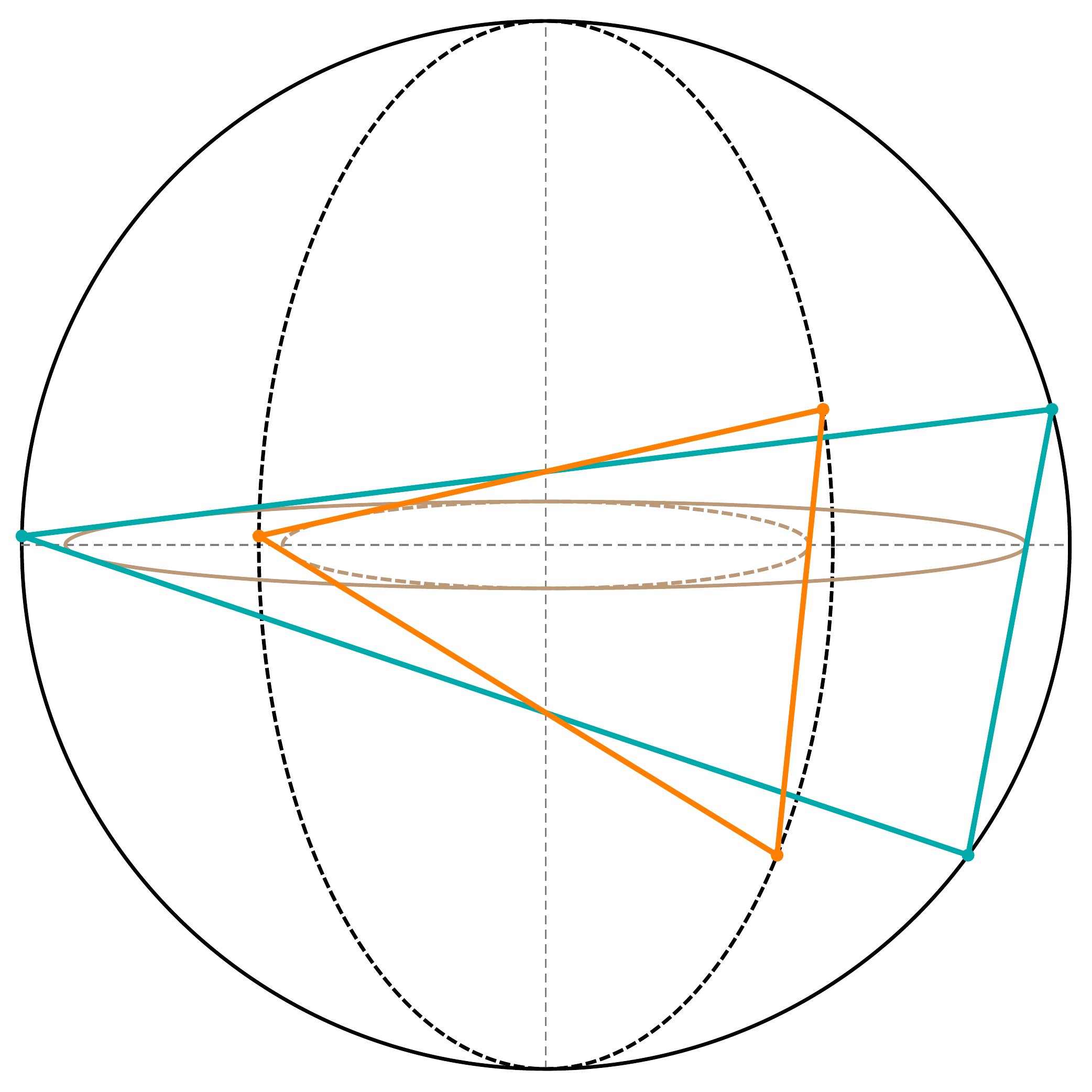}
    \caption{A 3-periodic (dark cyan) of the circumcircle family (solid black and brown ellipses) is shown along with a 3-periodic (orange) of the excentral family (dashed black and brown ellipses) such that the latter is a scaled image of the former along the major axis. Note both families are always acute; their angle triples sweep identical curves. \href{https://youtu.be/PMqoH4oGt10}{Video}}
    \label{fig:two-families-exc}
\end{figure}

A curious property of the orthic\footnote{A triangle's orthic has vertices at the feet of the altitudes \cite[Orthic Triangle]{mw}.} triangles to the circumcircle family is that both their inradius $r_h$ and circumradius $R_h$ are invariant \cite{garcia2020-family-ties}; see it in motion \href{https://bit.ly/30ZU0gz}{here}. Feuerbach proved that the product of cosines of a triangle is equal to $r_h/(4R_h)$  \cite[sec 299(g), p. 191]{johnson1960}. This entails that the circumcircle family conserves the product $k'$ of its angle cosines. In \cite[Lemma 1]{garcia2020-family-ties} the following expression is provided:

\begin{equation} k'=\prod_{i=1}^3{\cos\theta_i}=\frac{a b}{2 (a+b)^2}
\label{eqn:prod-cos}
\end{equation}

\begin{proposition}
The circumcircle family and its $s'$-scaled excentral image (\cref{lem:scaling-exc}) have identical products of cosines given by \cref{eqn:prod-cos}.
%\label{prop:cos-prod}
\end{proposition}

\begin{proof}
The product of excentral cosines of a generic triangle is given by $r/(4R)$ \cite[Circumradius]{mw}. So the product of cosines of the excentral family is given by a quarter of $(r/R)_\text{conf}$, \cref{eqn:rOvR}. We set $\alpha,\beta$ to the scaled caustic axes, i.e., $\alpha=s' a$ and $\beta=b$, where $s'$ is as in \cref{lem:scaling-exc}. After simplification obtain that $(1/4)(r/R)_\text{conf}$ is identical to \cref{eqn:prod-cos}.
\end{proof}

\section{Incircle and Confocals: Cosine Space}
\label{sec:cos-space-inc-conf}
Consider the usual parametrization for point on an ellipse $P(t)=[a\cos(t),b\sin(t)]$, i.e., the parameter $t$ for $P=(x,y)$ is $t=\tan^{-1}{(a y)/(b x)}$. Using it for a first vertex $P_1(t)$ of a 3-periodic family implies that the parameters for the other two vertices will be non-linear and distinct on $t$, see \cref{fig:02-jacobi-param}(top).

\begin{figure}
    \centering
    \includegraphics[width=\textwidth]{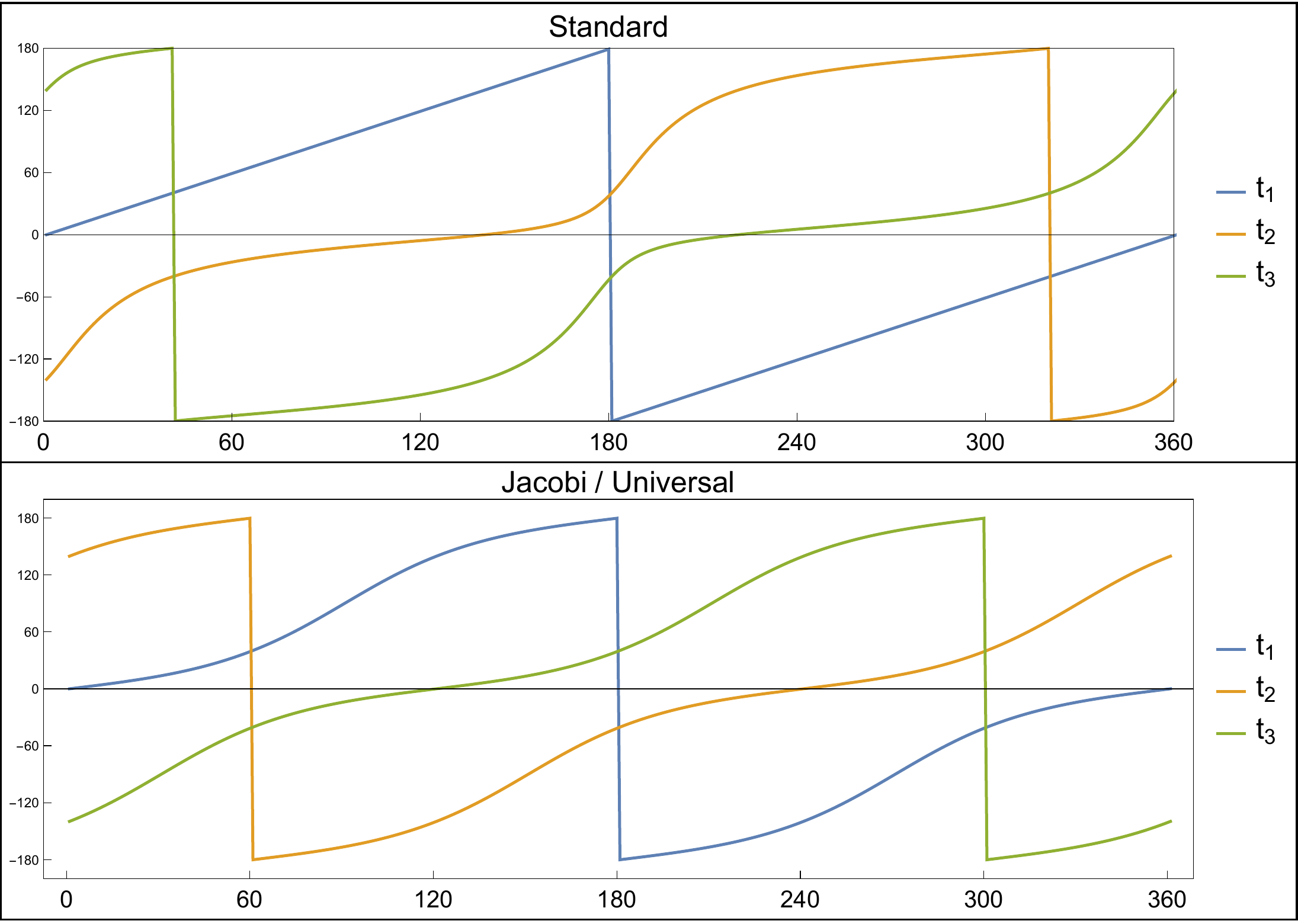}
    \caption{The parameter $t_i$ (vertical axis) of a point on an ellipse with semiaxes $a,b$ vs the parameter of a first vertex (horizontal axis). \textbf{Top:} usual parametrization, $P_1(t)=[a\cos{t},b\sin{t}]$. Notice while $P_1$'s position evolves linearly, those of $P_2$ and $P_3$ are different, non-linear curves. \textbf{Bottom:} Parametrization based on the universal-measure. The three curves are identical, separated by a 120-degree phase.} 
    \label{fig:02-jacobi-param}
\end{figure}

Fortunately, a special parametrization exists which renders the parameters identical functions (up to phase). This is based on the ``universal measure'' of the Poncelet map \cite{stachel2021-motion,koiller2021-spatial}. Namely, vertices are obtained by using fixed multiples of a constant $\Delta{u}$ as the argument of certain elliptic functions. Borrowing from \cite{stachel2021-motion}:

\begin{theorem*}
A billiard 3-periodic $P_i$ $(i=1,\ldots, N) $ of period $N$ with turning number $\tau$, where $\mathrm{gcd}(N,\tau) =1$,  is parametrized on $u$ with period $4K$ where:

\[ P_i=
\left[-a\,\sn  \left(u+ i \Delta{u},  {m} \right) , b\,\cn  \left(u + i \Delta{u}, { m} \right)\right] \]

where,
\[ m^2=\frac{a_c^2-b_c^2}{a_c^2},\;\;\Delta{u}=\frac{4\tau K}{N}\]
\[a= \sqrt{b^2+ a_c^2-b_c^2}, \;\; b=\frac{b_c}{\cn({\Delta{u}}/{2}, m)}\]
where $\sn,\cn,m$ are the elliptic sine, cosine, and modulus, respectively \cite{armitage2006}, $\tau$ is the turning number, and $K$ is the complete elliptic integral of the first kind.
\end{theorem*}

Since here we are considering billiard triangles, $N=3$ and $\tau=1$. As shown in \cref{fig:02-jacobi-param}(bottom), under this parametrization the parameters $t_i$ of the 3 vertices are identical functions at a 120-degree phase.

Recall the sum of cosines is constant for billiard 3-periodics. \cref{fig:02-jacobi-cos-param} shows how individual cosines in both the confocal and incircle families are either (i) 6 distinct functions under the usual parametrization, or (ii) the same exact ones (at different phases) under the Jacobi parametrization.

\begin{figure}
    \centering
    \includegraphics[width=\textwidth]{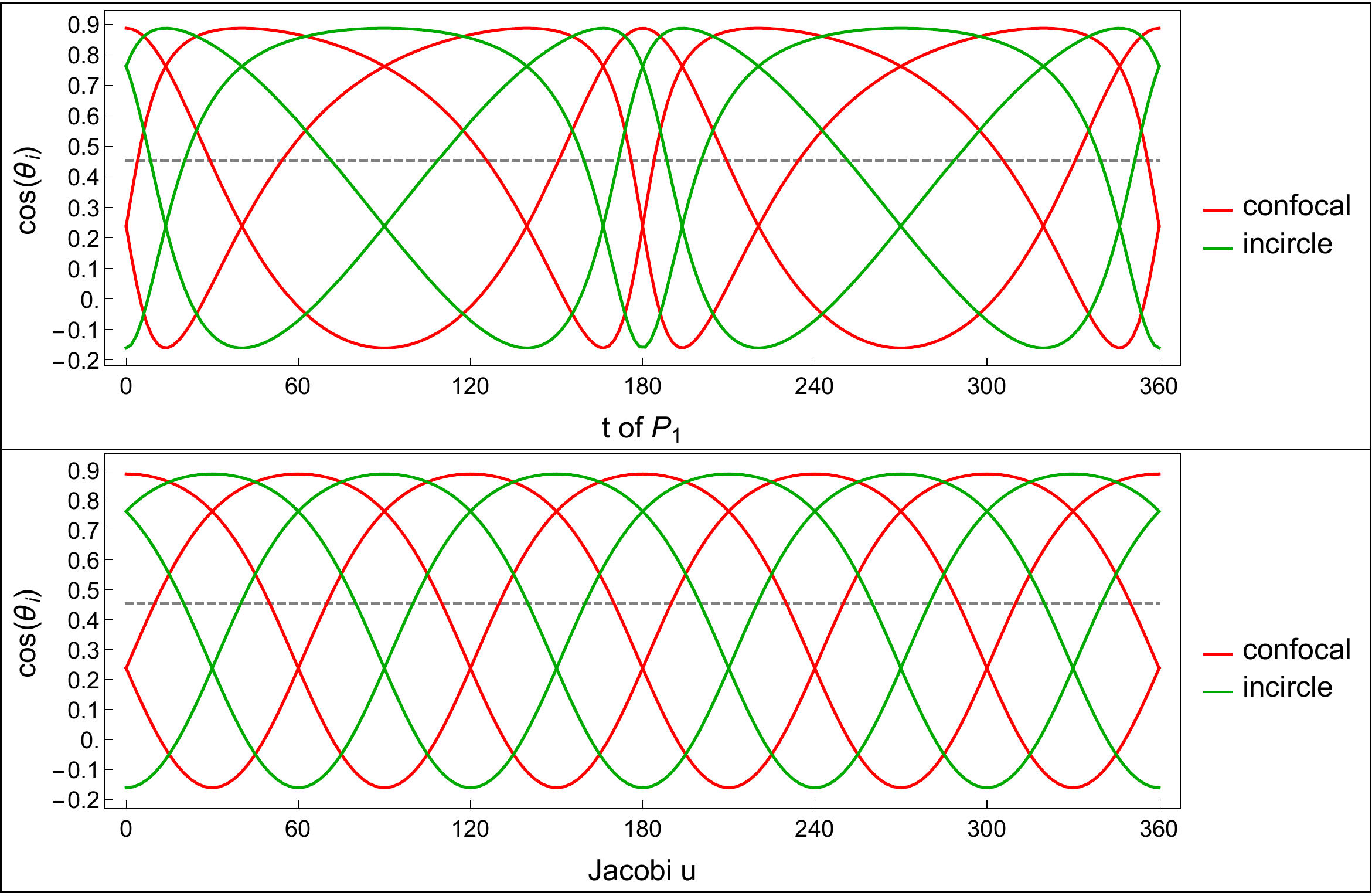}
    \caption{Cosines of billiard 3-periodics using the standard (top) vs. the universal measure parametrization (bottom) for both the confocal (red) and incircle (green) families. In the former case the six curves are distinct, while in the latter, cosines of a given family are at 60-degree phases, while the two sets are at a a 30 degree phase from each other. Also shown (dashed gray horizontal line) is the average cosine common to both families.}
    \label{fig:02-jacobi-cos-param}
\end{figure}
\index{parametrization!Jacobi}

\subsection*{Sweeping a Cubic}
 
 Let $k$ denote the invariant sum of cosines in the incircle family. As shown in \cref{fig:picks-3d}, the locus of the 3 cosines in 3d is a family of plane curves, since $c_1+c_2+c_3=k$. Referring to \cref{fig:proj-triples}:
 
 \begin{figure}
    \centering
    \includegraphics[width=\textwidth]{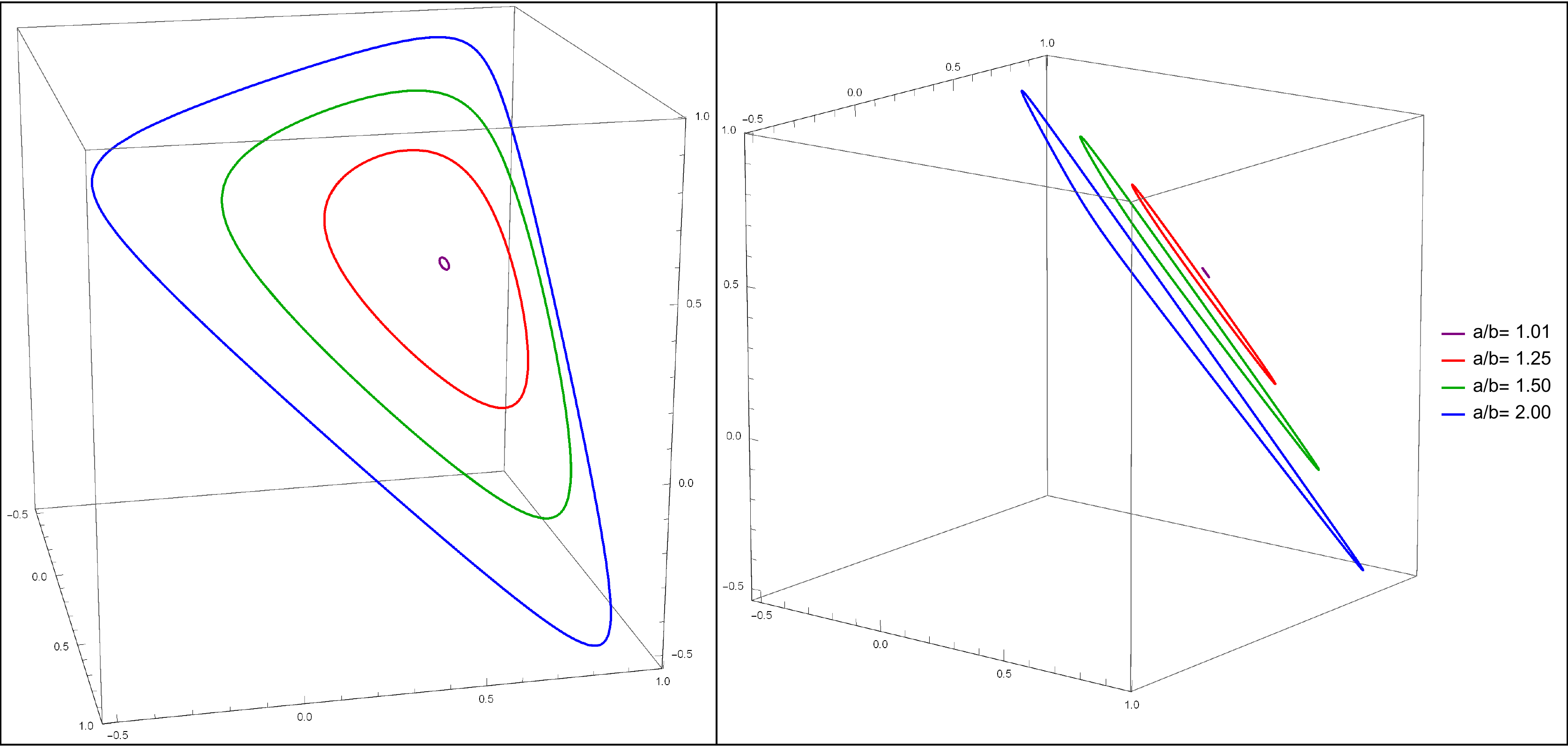}
    \caption{Two 3d views of the planar loci of cosines over 3-periodics in the confocal pair (and/or the affinely-related incircle pair), for various ratios of $a/b$. These are curves shaped as ``guitar picks'' which lie on the plane perpendicular to $[1,1,1]$.}
    \label{fig:picks-3d}
\end{figure}
 
\begin{theorem}
The locus $\Delta$ of cosine triples of 3-periodics in the incircle and affinely-related confocal families are equilateral cubics, given implicitly by:
\label{thm:incircle-confocal-space}
\[\Delta(u,v):3\sqrt{6}(3u^2 - v^2)v - 9(k - 3)(u^2 + v^2) + (2k-3)(k + 3)^2=0\]
where $k$ is as in \cref{prop:cos-sum}.

%\begin{align*}
%f(x,y,z)&=x+y+z-k =0\\
%    g(x,y)&= \left( a+b \right) ^{4}{x} y (x+y) - \left( a+b \right) ^{4}({x}^{2} +y^2)
%-2\, \left( {a}^{2}+3\,ab+{b}^{2}\right)  \left( a+b \right) ^{2}xy\\
%&+ \left( {a}^{2}+4\,ab+{b}^{2} \right)  \left( a+b \right) ^{2}(x +y)
%-2\,ab \left( {a}^{2}+3\,ab+{b}^{2} \right) =0\\
%&=(y - 1)(x - 1)(x + y)a^4 - 2r(y - 1)(x - 1)a^3 + 2((y - 1)x - y)r^2a^2 + 4ar^3 - 2r^4=0\\
% &=2xy(x+y) - 2(x^2 +  y^2)- 2(k +1)xy + 2k(x+y)   +1- k^2 =0\\
%&k={\frac {{a}^{2}+4\,ab+{b}^{2}}{ \left( a+b \right) ^{2}}},\;\; r=\frac{ab}{a+b}
%\end{align*}
\end{theorem}

\begin{proof}
Let $P_i$ be the vertices of a 3-periodic in the incircle family. Parametrize $P_1(t)=[a\cos{t},b\sin{t}]$. Using trigonometry, derive the following expressions for $P_2(t)$ and $P_3(t)$:

\begin{align*}
P_2&=w_2 \left[a^2  \cos{t} - w_1 \sin{t}, b^2  \sin{t} + w_1 \cos{t} \right]\\
P_3&=w_2 \left[a^2  \cos{t} + w_1 \sin{t}, b^2  \sin{t} - w_1 \cos{t}\right]
\end{align*}
where $w_1=\sqrt{c^2 (a + b)^2  \cos^2{t} + 2 a b^3 + b^4}$ and $w_2=-a b/(({c^2} \cos^2{t} + b^2) (a + b))$.

From the above, apply the law of cosines to obtain the following expressions for the $c_i=\cos\theta_i$: 

\begin{align*}
c_1&= \frac{(a^4 + 2 a^3 b - 2 a b^3 - b^4)\cos^2{t} - a^2 b^2 + 2 a b^3 + b^4}{(a + b)^2 (c^2 \cos^2{t} + b^2)} \\
c_2^2&= \frac{2 ((b-a)  \cos^2{t} - b) a b (a-b) w_1  \sin{t} \cos{t} - c^6  \cos^6{t} + w_3 \cos^4{t} +  w_4 }{(a + b)^2 ({c^2}  \cos^2{t} + b^2)^2}\\
c_3^2 &= \frac{2 ((a - b) \cos{t}^2 + b) a b (a - b) w_1 \sin{t}\cos{t} - c^6 \cos^6{t} + w_3 \cos^4{t} + w_4}{(a + b)^2 (c^2 \cos{t}^2 + b^2)^2}
\end{align*}
where $w_3=(a^6 - 2 a^4 b^2 - 4 a^3 b^3 + 7 a^2 b^4 - 2 b^6)$ and $w_4=(4 a^3 b^3 - 5 a^2 b^4 + b^6)  \cos^2{t} + a^2 b^4$. Using $c_3=k-c_1-c_2$ via CAS-assisted theory of resultants eliminate sines and $\sin{t}$ and $\cos{t}$, obtaining the following implicit on $c_1$ and $c_2$:

\begin{equation}
2c_1 c_2(c_1+c_2) - 2(c_1^2 + c_2^2)- 2(k +1)c_1 c_2 + 2k(c_1+c_2) +1- k^2 =0
\label{eq:implicit-pick}
\end{equation}

Intersect the above extrusion with the $c_1+c_2+c_3=k$ plane and use the basis $u,v$, $u=([0,0,1]{\times}[1,1,1])/||.||$, and $v=([1,1,1]{\times}u)/||.||$.

To prove that the above is equal to the locus obtained in the confocal pair, apply the affine transformation to the $P_i$ using $s$ as in \cref{lem:scaling}. Carry out the same steps and obtain the exact same implicit as in \cref{eq:implicit-pick}.
\end{proof}

\begin{figure}
    \centering
    \includegraphics[width=.66\textwidth]{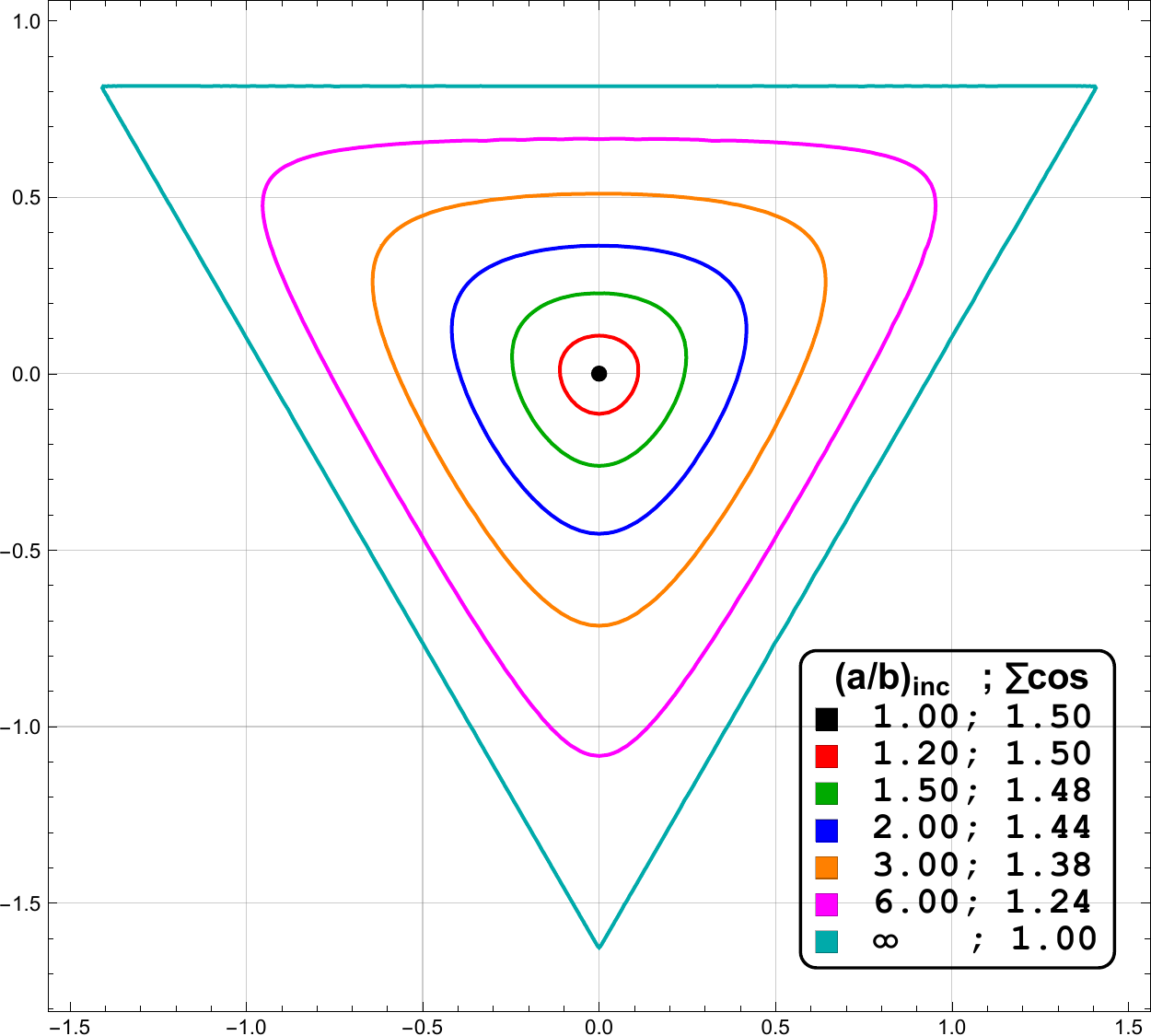}
    \caption{Projections of cosine triples for the incircle family onto the constant-sum-of-cosines plane for various ratios of $a/b$ of the external ellipse is a family of equilateral cubics. When $a/b$ goes to zero (resp. infinity), the locus collapses to a point (resp. tends to an equilateral triangle whose vertices are $2\sqrt{2}/3$ units from its centroid.}
    \label{fig:proj-triples}
\end{figure}

%\begin{figure}
%    \centering
%    \includegraphics[width=\textwidth]{pics/0030_picks_2d.pdf}
%    \caption{Picks 2d. \href{https://youtu.be/uwdW95HI-q8}{Video}}
%    \label{fig:picks-2d}
%\end{figure}

The common locus of the incircle and confocal cosine triples is illustrated in \cref{fig:both}(left). 

Straightforward derivation yields the minimal and maximal values $c_{min},c_{max}$ of $\cos\theta_i$ in terms of $a,b$ of the incircle family:

\[ \left\{c_{min}, c_{max}\right\}=\left\{1-\frac{2a^2}{(a+b)^2},\,1-\frac{2b^2}{(a+b)^2}\right\}\]

%\section{Circumcircle and Excentral Product of Cosines}
%\label{sec:prod-cos}
%\input{035_prod_cos}

\section{Circumcircle and Excentrals: Log-Cosine Space}
\label{sec:cos-space-circ-exc}
Let $k'=\prod{c_i}'$ denote the invariant product of cosines in the circumcircle family. Therefore in ``log cosine space'' a planar curve is swept, shown superposed with the locus of incircle-confocal cosines in \cref{fig:superposed}. As before, due to periodicity, all $c_i'$ sweep the same function though out-of-phase. 
 
 %\textcolor{red}{ronaldo: universal measure?}

\begin{figure}
    \centering
    \includegraphics[width=.7\textwidth]{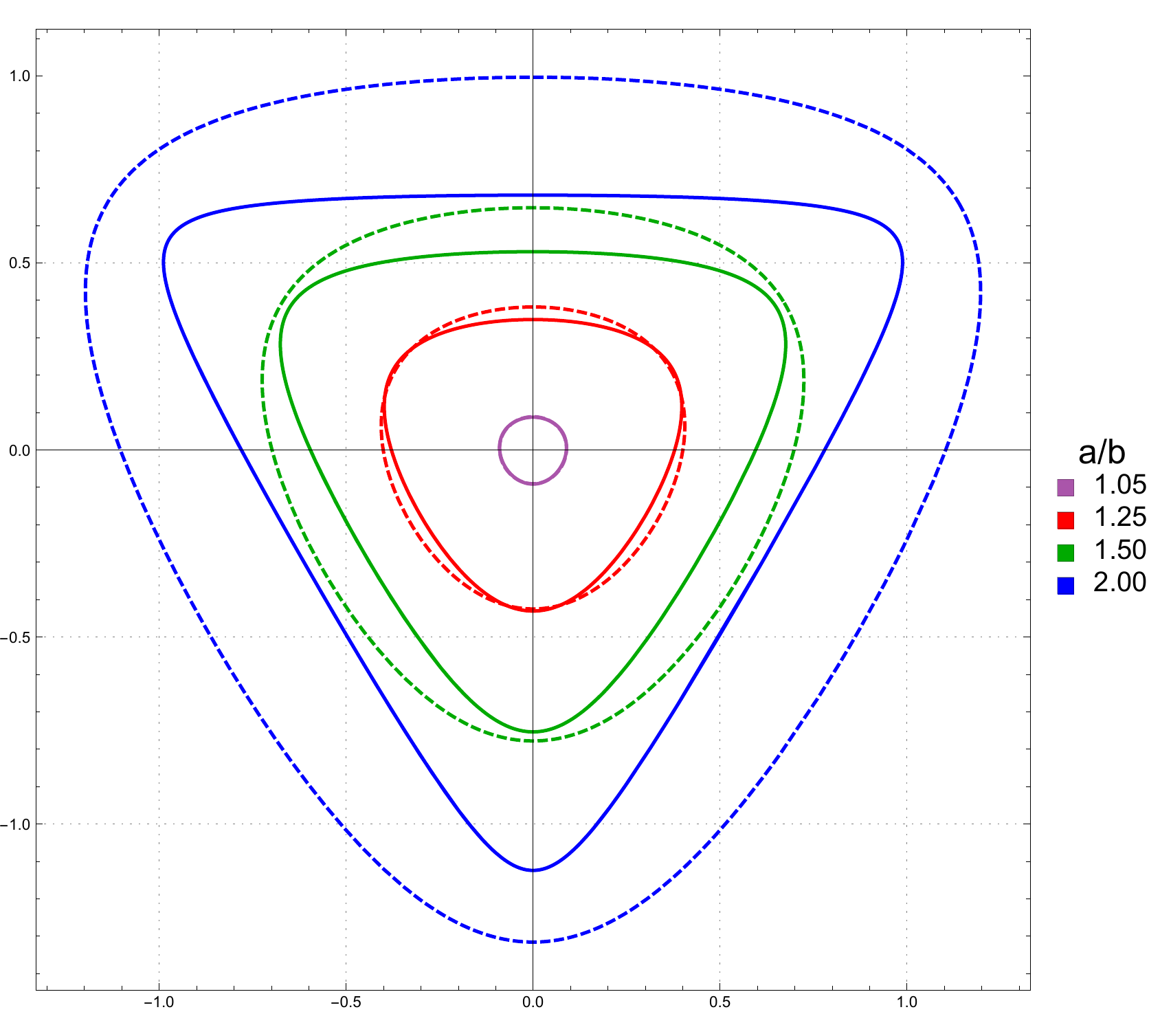}
    \caption{\textbf{Solid curves:} locus of cosine triples over incircle and confocal families. \textbf{Dashed curves:} locus of log-cosines over circumcircle and excentral families. A family of curves is shown for various aspect ratios of the external confocal ellipse or the caustic for the excentral families.}
    \label{fig:superposed}
\end{figure}
 
 The implicit $x y z = \zeta$ where $\zeta$ is a constant is known as the Titeica surface \cite{ferreol2017-titeica}. Referring to \cref{fig:titeica}:
 
 \begin{theorem}
The locus $\Delta'$ of cosine triples over 3-periodics in the circumcircle pair (or its affinely-related excentral family) is the intersection of a sphere with a Titeica surface, i.e.:
\[ \Delta':\;x^2+y^2+z^2+2k'-1=0,\;\;\;x y z - k'=0\]
\label{thm:locus-exc}
\end{theorem}

\begin{proof}
Let $P_i'$ be the vertices of a 3-periodic in the circumcircle family. Parametrize $P_1'(t)=(a+b)\;[u,\sqrt{1-u^2}]$. Using trigonometry, derive the following expressions for $P_2'(t)$ and $P_3'(t)$:

\begin{align*}
P_2'&=\frac{1}{\left( {u}^{2}-1 \right) {a}^{2}-{
b}^{2}u^2}\left[  \left( {b}^{2}u-\sqrt {1-{u}^{2} }w \right) a, \left( \sqrt 
{1-{u}^{2}}{a}^{2}+ w u \right) b
\right]\\
P_3'&=\frac{1}{\left( {u}^{2}-1 \right) {a}^{2}-{
b}^{2}u^2}\left[  \left( {b}^{2}u+w \sqrt {1-{u}^{2}}\right) a, \left( \sqrt {1-{u}^{2}}{a}^{2}-w u
 \right) b]
\right]\\
w&= \sqrt {{a}^{3}
 \left( a+2\,b \right) -{c}^{2} \left( a+b \right) {u}^{2}}
\end{align*}

From the above, apply the law of cosines to obtain the following expressions for the $c_i'=\cos\theta_i'$: 

\begin{align*}
(c_1')^2&= {\frac {{a}^{2}{b}^{2}}{ \left( a+b \right) ^{2} \left(  \left(1- {u}^{
2}  \right) {a}^{2}+{b}^{2}{u}^{2} \right) }}
\\
(c_2')^2&=\frac{1}{2}\,{\frac { \left( {u}^{2}-
1 \right) {a}^{3}-{b}^{3}{u}^{2}+\sqrt {1-{u}^{2}}\,  \left( a-b \right) u w }{ \left( a+b \right)  \left(  \left( 
{u}^{2}-1 \right) {a}^{2}-{b}^{2}{u}^{2} \right) }}
\\
(c_3')^2&= \frac{1}{2} \,{\frac {  \left( {u}^{2}
-1 \right) {a}^{3}-{b}^{3}{u}^{2} -\sqrt {1-{u}^{2}}\,  \left( a-b \right)u w}{ \left( a+b \right)  \left( 
 \left( {u}^{2}-1 \right) {a}^{2}-{b}^{2}{u}^{2} \right) }}
\end{align*}

Therefore,

\[\sum{(c_i')^2}=\frac{a^2+ab+b^2}{(a+b)^2}=1-2k'.\]

Intersect the above sphere with the Titeica surface $c_1 c_2 c_3=k'$. Finally, to prove the above is equal to the locus obtained in the excentral pair, apply the affine transformation to the $P_i'$ using $s'$ as in \cref{lem:scaling-exc}. Carry out the same steps and obtain the exact same two surfaces defining $\Delta'$.
\end{proof}

The common locus of the circumcircle and excentral log-cosine triples is illustrated in \cref{fig:both}(right).

Note one can eliminate $k'$ from \cref{thm:locus-exc} to obtain an implicit for the union of all spherical curves swept by cosine triples over the circumcircle (or excentral) family. This is given by:

\[ 2 x y z + x^2 + y^2 + z^2 = 1 \]

A few spherical loci superposed on the positive octant of the above are shown in \cref{fig:sfiha}.

\begin{figure}
    \centering
    \includegraphics[width=.5\textwidth]{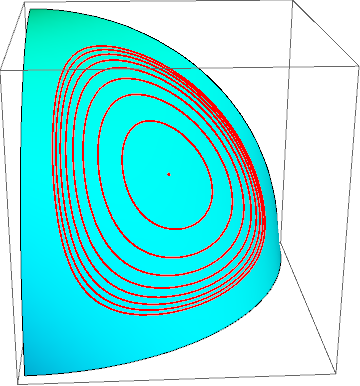}
    \caption{A few spherical loci of cosine triples over the circumcircle/excentral families shown superposed on the surface they sweep.}
    \label{fig:sfiha}
\end{figure}

%Recall the scaling parameter $s'$ defined in \cref{lem:scaling-exc} which sends the circumcircle pair to the excentral pair, whose internal ellipse axes are $a'=s' a$ and $b'=b$ as well as the external ellipse axis $a_c=s'R$ and $b_c=R$.

\begin{figure}
    \centering
    \includegraphics[width=\textwidth]{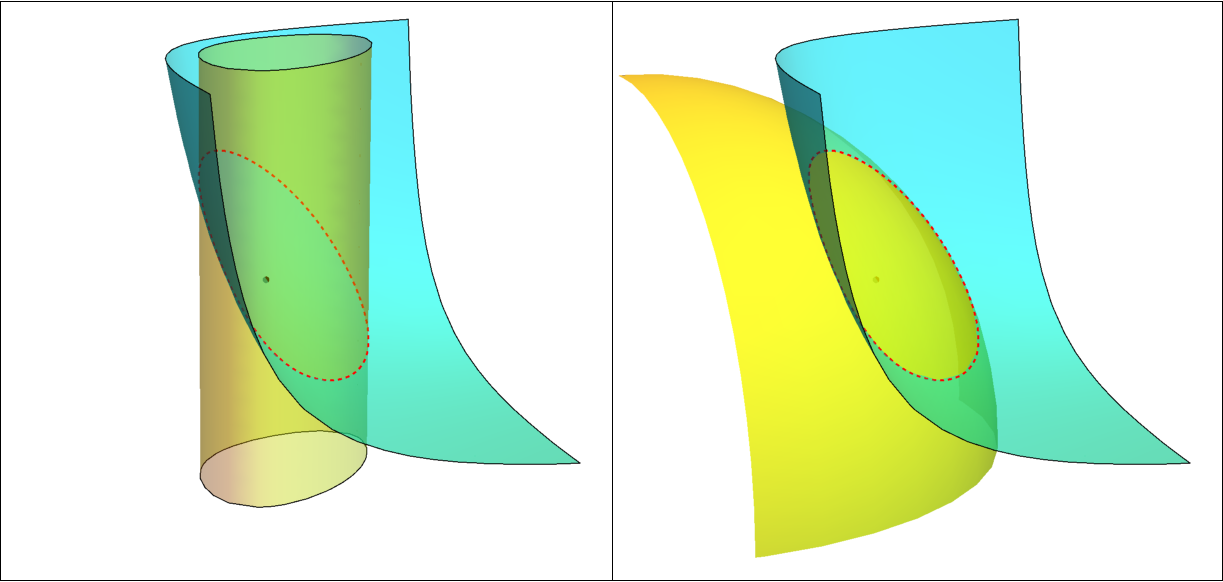}
    \caption{Cosine triples over both circumcircle and excentral families sweep a spherical curve which can be obtained as the intersection of the Titeica surface (light blue) \cite{ferreol2017-titeica} with either (i) a cylindrical surface (left) or (ii) a sphere (right). In either graph a small dot marks the intersection of the Titeica surface with the line from the origin toward $[1,1,1]$.}
    \label{fig:titeica}
\end{figure}

%\begin{corollary}
%\textcolor{red}{is this possible}
%The locus of the logarithm of cosine triples of 3-periodics in the circumcircle (or its affinely-related excentral) family is a planar curve. It given implicitly as:
%\[ f(x,y)=0 \]
%\label{cor:circum-exc-space}
%\end{corollary}

%\begin{proof}
%\end{proof}

%\begin{figure}
%    \centering
%    \includegraphics[width=\textwidth]{pics/0010_cos_graph.pdf}
%    \caption{\textbf{Top:} evolution of cosine triples for confocal and incircle families (red,green) vs $t$ used to parametrize $P_1=[a\cos{t},b\sin{t}]$. Note that sum (and average, dashed grey) are the same. \textbf{Bottom:} the same plot, but now considering the triple of ``log cosines'' in the excentral and incircle families (orange, light blue). Note both curves also produce the same sum and average (dashed grey).}
%    \label{fig:cos-graph}
%\end{figure}

\begin{figure}
    \centering
    \includegraphics[width=\textwidth]{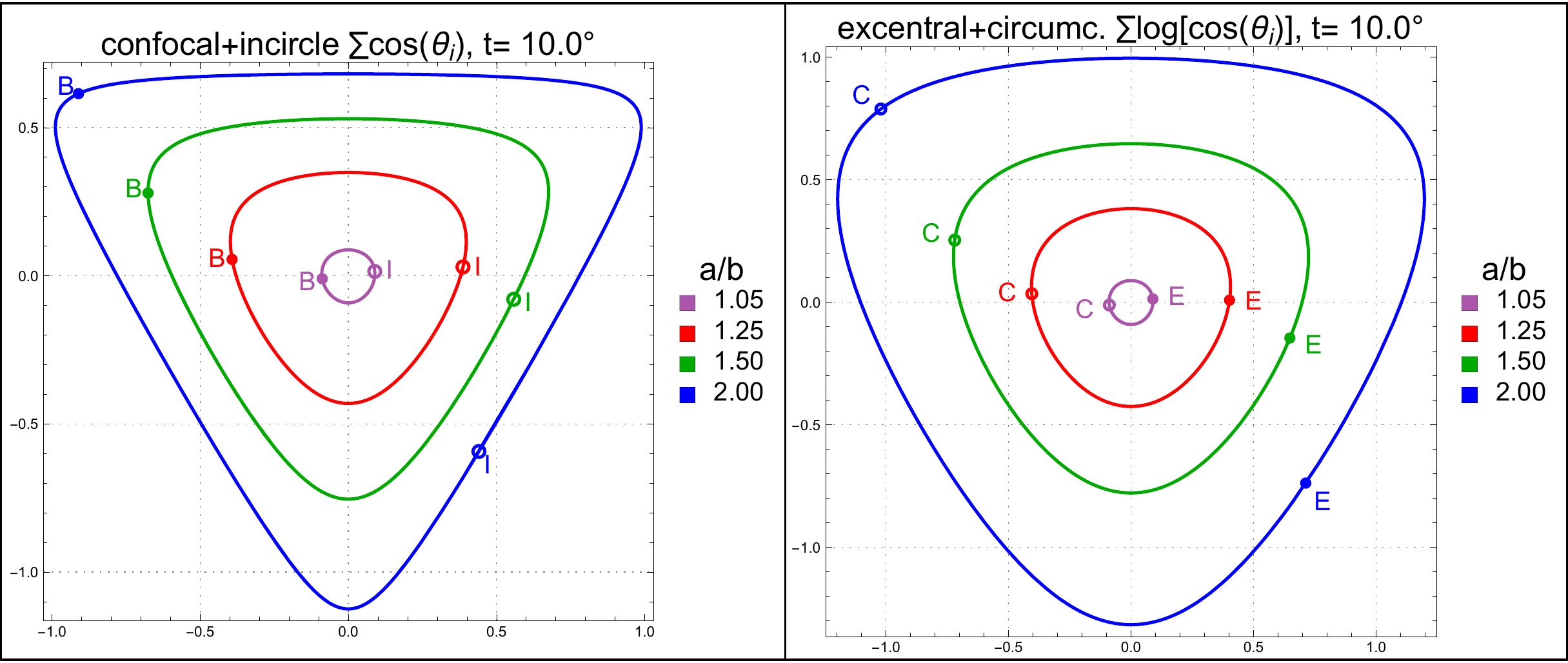}
    \caption{\textbf{Left:} Cosine vector of 3-periodics over the billiard/confocal (B) families and the affinely-related incircle (I) families, projected onto the plane perpendicular to $[1,1,1]$, over several confocal axis ratios $a/b$ of the confocal pair. The key observation is that they sweep identical plectrum-shaped curves, though out-of-phase \href{https://youtu.be/uwdW95HI-q8}{Video}. \textbf{Right:} The same visualization for the log of cosines for the excentral (E) and circumcircle (C) 3-periodic families. As shown, for each $a/b$ of the excentral caustic, both curves swept are identical, rounder, plectrum-shaped curves, where (E) and (C) are out-of-phase.}
    \label{fig:both}
\end{figure}

Straightforward derivation yields the minimal and maximal values $c_{min},c_{max}$ of $\cos\theta_i$ over the circumcircle family:

\[   \left\{c_{min}, c_{max}\right\}=\left\{\frac{b}{R},\,\frac{a}{R}\right\}=\left\{\frac{b}{a+b},\,\frac{a}{a+b}\right\}\]

%\section{Limiting Pedals}
%\label{sec:limiting-pedals}
%\input{050_limiting_pedals}

\section{conclusion}
\label{sec:conclusion}
Animations illustrating some phenomena in this article are listed on Table~\ref{tab:playlist}.

\begin{table}[H]
\small
\begin{tabular}{|c|l|l|}
\hline
id & Title & \textbf{youtu.be/<.>}\\
\hline
01 & {N=3 Poncelet family} &
\href{https://youtu.be/bjHpXVyXXVc}{\texttt{bjHpXVyXXVc}}\\
02 & {N=7 Poncelet family} &
\href{https://youtu.be/kzxf7ZgJ5Hw}{\texttt{kzxf7ZgJ5Hw}}\\
03 & {Affine images of confocal 3-periodics} &
\href{https://youtu.be/aWs29dqY34M}{\texttt{aWs29dqY34M}}\\
04 & {Affine images of confocal 5-periodics} &
\href{https://youtu.be/buhDTsjpvRQ}{\texttt{buhDTsjpvRQ}}\\
05 & {Incircle and Affinely-Related Confocal Families} &
\href{https://youtu.be/CKVoQvErjj4}{\texttt{CKVoQvErjj4}}\\
06 & {Circumcircle and Affinely-Related Excentral Families} &
\href{https://youtu.be/PMqoH4oGt10}{\texttt{PMqoH4oGt10}}\\
07 & {Loci of Cosines for Incircle and Confocal Families} &
\href{https://youtu.be/uwdW95HI-q8}{\texttt{uwdW95HI-q8}}\\
\hline
\end{tabular}
\caption{Animations of some phenomena. The last column is clickable and provides the YouTube code.}
\label{tab:playlist}
\end{table}

A few questions are posed to the reader:

\begin{itemize}
    \item The cosine triples in the incircle/confocal families sweep a plane curve. Recall their $N>3$ counterparts also conserve cosines \cite{akopyan2020-invariants}. In which dimension $d$ will their cosine $N$-tuples live, namely, can $d<N-1$?
    \item Similarly, for $N>3$, will the cosine tuples in the circumcircle/excentral lie on an $(N-1)$-sphere or be even more constrained?
    \item What is the implicit equation for the planar log-cosine curve swept by the excentral and/or circumcircle family?

    \item For the $N>3$ case, will the confocal family (resp.\ excentral) and its affine image with incircle (resp.\ circumcircle) sweep the same curve in cosine (resp.\ log-cosine) space? 
  
\end{itemize}

\section*{Acknowledgements}
\noindent We would like to thank A. Akopyan for pointing out the relationship of some of our results with \cite[Thm 6.4]{akopyan2020-invariants}.

%\appendix

%\section{Jacobi's Elliptic Functions}
%\input{130_app_jacobi}

%\section{Proof to  invariant sum of cosines transformation}
%\input{130_proof_thm_scale}
%\label{app:Proof_Scale}

%\section{Proof to invariant product of cosines transformation}
%\input{140_proof_thm_scale2}
%\label{app:Proof_Scale_Product}

%\section{Table of Symbols}
%\input{120_app_symbols}
%\label{app:symbols}

\bibliographystyle{maa}
\bibliography{references,authors_rgk_v3}

\begin{thebibliography}{10}
\expandafter\ifx\csname urlstyle\endcsname\relax
 \providecommand{\url}[1]{doi:\discretionary{}{}{}#1}\else
 \providecommand{\url}{doi:\discretionary{}{}{}\begingroup
  \urlstyle{rm}\Url}\fi

\bibitem{akopyan2020-invariants}
Akopyan, A., Schwartz, R., Tabachnikov, S. (2020).
\newblock Billiards in ellipses revisited.
\newblock \emph{Eur. J. Math.}
\newblock {d}oi:10.1007/s40879-020-00426-9.

\bibitem{armitage2006}
Armitage, J.~V., Eberlein, W.~F. (2006).
\newblock \emph{Elliptic Functions}.
\newblock London: Cambridge University Press.

\bibitem{bialy2020-invariants}
Bialy, M., Tabachnikov, S. (2020).
\newblock {Dan Reznik's} identities and more.
\newblock \emph{Eur. J. Math.}
\newblock {d}oi:10.1007/s40879-020-00428-7.

\bibitem{bos1987}
Bos, H. J.~M., Kers, C., Oort, F., Raven, D.~W. (1987).
\newblock Poncelet's closure theorem.
\newblock \emph{Expositiones Math.}, 5(4): 289--–364.

\bibitem{centina15}
del Centina, A. (2016).
\newblock {P}oncelet's porism: a long story of renewed discoveries i.
\newblock \emph{Arch. Hist. Exact Sci.}, 70(2): 1--122.

\bibitem{caliz2020-area-product}
Chavez-Caliz, A. (2020).
\newblock More about areas and centers of {Poncelet} polygons.
\newblock \emph{Arnold Math J.}
\newblock {d}oi:10.1007/s40598-020-00154-8.

\bibitem{cieslak2013}
Cieślak, W., Martini, H., Mozgawa, W. (2013).
\newblock On the rotation index of bar billiards and {P}oncelet's porism.
\newblock \emph{Bull. Belg. Math. Soc. Simon Stevin}, 20(2): 287--300.

\bibitem{connes07}
Connes, A., Zagier, D. (2007).
\newblock A property of parallelograms inscribed in ellipses.
\newblock \emph{The American Mathematical Monthly}, 114(10): 909--914.

\bibitem{dragovic11}
Dragovi\'{c}, V., Radnovi\'{c}, M. (2011).
\newblock \emph{Poncelet Porisms and Beyond: Integrable Billiards,
  Hyperelliptic Jacobians and Pencils of Quadrics}.
\newblock Frontiers in Mathematics. Basel: Springer.

\bibitem{ferreol2017-titeica}
Ferréol, R. (2017).
\newblock The {T}iteica surface.
\newblock MathCurve.
\newblock \url{https://mathcurve.com/surfaces.gb/titeica/titeica.shtml}.

\bibitem{flatto2009}
Flatto, L. (2009).
\newblock \emph{Poncelet's theorem}.
\newblock Am. Math. Soc., Providence, RI.
\newblock Chapter 15 by S. Tabachnikov.

\bibitem{garcia2019-incenter}
Garcia, R. (2019).
\newblock Elliptic billiards and ellipses associated to the 3-periodic orbits.
\newblock \emph{American Mathematical Monthly}, 126(06): 491--504.

\bibitem{garcia2020-family-ties}
Garcia, R., Reznik, D. (2021).
\newblock Family ties: Relating {P}oncelet 3-periodics by their properties.
\newblock \emph{J. Croatian Soc. for Geom. \& Gr. (KoG)}, to appear.
\newblock {arXiv}:2012.11270.

\bibitem{garcia2020-new-properties}
Garcia, R., Reznik, D., Koiller, J. (2020).
\newblock New properties of triangular orbits in elliptic billiards.
\newblock \emph{Amer. Math. Monthly}, to appear.
\newblock {arXiv}:2001.08054.

\bibitem{georgiev2012-cayley}
Georgiev, V., Nedyalkova, V. (2012).
\newblock Poncelet’s porism and periodic triangles in ellipse.
\newblock \emph{Dynamat}.
\newblock \url{www.dynamat.oriw.eu/upload_pdf/20121022_153833__0.pdf}.

\bibitem{griffiths1978}
Griffiths, P., Harris, J. (1978).
\newblock On {C}ayley's explicit solution to {P}oncelet's porism.
\newblock \emph{Enseign. Math. (2)}, 24(1-2): 31--40.

\bibitem{johnson1960}
Johnson, R.~A. (1960).
\newblock \emph{Advanced Euclidean Geometry: An Elementary Treatise on the
  Geometry of the Triangle and the Circle}.
\newblock New York, NY: Dover, 2nd ed.
\newblock Editor John W. Young.

\bibitem{kaloshin2018}
Kaloshin, V., Sorrentino, A. (2018).
\newblock On the integrability of {B}irkhoff billiards.
\newblock \emph{Phil. Trans. R. Soc.}, A(376).

\bibitem{koiller2021-spatial}
Koiller, J., Reznik, D., Garcia, R. (2021).
\newblock Average elliptic billiard invariants with spatial integrals.
\newblock arXiv:2102.10899.

\bibitem{levi2007-poncelet-grid}
Levi, M., Tabachnikov, S. (2007).
\newblock The {P}oncelet grid and billiards in ellipses.
\newblock \emph{Am. Math. Monthly}, 114(10): 895--908.

\bibitem{lynch2019}
Lynch, P. (2019).
\newblock Integrable elliptic billiards and ballyards.
\newblock \emph{Eur. J. of Phys.}, 41(1).

\bibitem{reznik2020-intelligencer}
Reznik, D., Garcia, R., Koiller, J. (2020).
\newblock Can the elliptic billiard still surprise us?
\newblock \emph{Math Intelligencer}, 42: 6--17.

\bibitem{reznik2021-fifty}
Reznik, D., Garcia, R., Koiller, J. (2021).
\newblock Fifty new invariants of {N}-periodics in the elliptic billiard.
\newblock \emph{Arnold Math. J.}, 7: 341--355.

\bibitem{rozikov2018}
Rozikov, U.~A. (2018).
\newblock \emph{An Introduction To Mathematical Billiards}.
\newblock Hackensack, NJ: World Scientific Publishing Co.

\bibitem{schwartz2007-grid}
Schwartz, R.~E. (2007).
\newblock The {P}oncelet grid.
\newblock \emph{Advances in Geometry}, 7(2): 157--175.

\bibitem{stachel2021-grid}
Stachel, H. (2021).
\newblock The geometry of billiards in ellipses and their {P}oncelet grids.
\newblock {arXiv}:2105.03362.

\bibitem{stachel2021-iso}
Stachel, H. (2021).
\newblock Isometric billiards in ellipses and focal billiards in ellipsoids.
\newblock {arXiv}:2105.05295.

\bibitem{stachel2021-motion}
Stachel, H. (2021).
\newblock On the motion of billiards in ellipses.
\newblock {arXiv}:2105.03624.

\bibitem{sergei91}
Tabachnikov, S. (2005).
\newblock \emph{Geometry and Billiards}, vol.~30 of \emph{Student Mathematical
  Library}.
\newblock Providence, RI: American Mathematical Society.
\newblock Mathematics Advanced Study Semesters, University Park, PA.

\bibitem{mw}
Weisstein, E. (2019).
\newblock Mathworld.
\newblock \emph{MathWorld--A Wolfram Web Resource}.
\newblock \url{mathworld.wolfram.com}.

\end{thebibliography}

\end{document}